\documentclass{article}

\usepackage[left=2.5cm, right=2cm, top=3cm, bottom=1cm]{geometry}
\usepackage[shortlabels]{enumitem}
\usepackage{amsmath,amsthm,amssymb,fancyhdr,graphicx}
\usepackage{tikz}
\usetikzlibrary{positioning,graphs,graphs.standard}

\newcommand{\family}[1]{\mathcal{#1}}    
\newcommand{\partialG}{\family{A}}       
\newcommand{\fullG}{\family{B}}          

\newtheorem{theorem}{Theorem}
\newtheorem{lemma}{Lemma}

\newtheorem{prop}{Proposition}

\newcounter{enumTemp}

\begin{document}
\begin{flushleft}
\pagestyle{fancy}

\begin{center}
\huge
Non-Hamiltonian 2-regular Digraphs \\
\large
Munagala V. S. Ramanath\\
\end{center}

\begin{abstract}
  In earlier papers, we showed a decomposition of 2-diregular digraphs ({\sl 2-dds}) and
  used it to provide some sufficient conditions for these graphs to be non-Hamiltonian;
  we also showed a close connection between the permanent and determinant of the
  adjacency matrices of these digraphs and gave some enumeration and generation results.
  In the present paper we extend the discussion
  to a larger class of digraphs, introduce the notions of routes and quotients and use
  them to provide additional criteria for 2-dds to be non-Hamiltonian. Though individual
  non-Hamiltonian regular connected graphs of low degree are known (e.g. Tutte and
  Meredith graphs), families of such graphs are not common in the literature; even
  scarcer are families of such digraphs. Our results  identify a few such families.
\end{abstract}

\section{Introduction}
We explore a class of 2-regular digraphs by providing some necessary and some sufficient
conditions for non-Hamiltonicity by building on the results and decomposition of the
arc-set into {\em alternating cycles} outlined in [\ref{ref-ram}].

\section{Definitions}
\label{sec-def}
Let $G = (V,A)$ be a digraph. As usual, we say that a subgraph $H$ is a {\em component}
if it is a connected component of the underlying (undirected) graph. We use $c(G)$ to
denote the number of components of $G$.
For an arc $e = (u, v)$ of $G$, $u$ the {\it start-vertex} and $v$ the {\it end-vertex}
of $e$; $u$ is a {\em predecessor} of $v$, $v$ is a {\em successor} of $u$, $e$ is an
\emph{in-arc} of $v$ and \emph{out-arc} of $u$.

For a subset $U \subset V$, we use $U_{in}$ and $U_{out}$ for the set of arcs entering
[exiting] some vertex of $U$ from [to] $V - U$. When $U = \{v\}$, we write $v_{in}$ and
$v_{out}$.

$G$ is called a {\it k-digraph} iff for every vertex $v$ we have:
\begin{align*}
  |v_{in}| &= k \; {\rm and} \; |v_{out}| = 0; \; {\rm or} \\
  |v_{in}| &= 0 \; {\rm and} \; |v_{out}| = k; \; {\rm or} \\
  |v_{in}| &= |v_{out}| = k
\end{align*}

Suppose $F$ is a 1-digraph. Clearly, it must be a (disjoint) collection of components
each of which is either a simple cycle or a simple non-cyclic path; let $F_c$ and $F_p$
denote, respectively, the sets of these components. The {\em index} of $F$ denoted by
$i(F)$ is the number of cycles in it, i.e. $|F_c|$. We call $F$ {\it open} if
$i(F) = 0$, {\it closed} otherwise. If $F_p = \emptyset$, $F$ can be viewed as a
permutation and, in this case, $F$ is {\em even} or {\em odd} according as the
permutation is even or odd. As noted in Lemma 1 of [\ref{ref-ram}] $F$ is even iff
$|V|$ and $i(F)$ have the same parity.

Suppose $G$ is a k-digraph. A 1-digraph that is a spanning subgraph of $G$ will be called
a {\em factor} of $G$ ({\em difactor} is more appropriate but since all graphs in
this document are directed, we use the shorter term for brevity).

$V$ can be partitioned into three disjoint subsets, based on the in- and out-degree at
each vertex, called respectively the set of {\it entry}, {\it exit} and {\it saturated}
vertices:
\begin{align*}
  V_{entry} &= \{ v \in V \; | \; |v_{in}|  = 0, |v_{out}| = k \}\\
  V_{exit}  &= \{ v \in V \; | \; |v_{in}|  = k, |v_{out}| = 0 \}\\
  V_{sat}   &= \{ v \in V \; | \; |v_{in}| = |v_{out}| = k \}
\end{align*}

The entry and exit vertices are {\it unsaturated} vertices.
$G$ is {\em closed} if {\em all} of its factors are {\em closed}, {\em open} otherwise.
A Hamiltonian circuit of $G$ is clearly a (closed) factor of $G$.

A 1-digraph $F$ where $V_{entry} \neq \emptyset$ defines a unique bijection $r_F$ from
$V_{entry}$ to $V_{exit}$ where $r_F(u)$ is the unique exit vertex of $F$ where the simple
path beginning at $u$ ends. This bijection will be called the {\em route} defined by
$F$; it is {\em open} or {\em closed} according as $F$ is open or closed. Notice that
the route is defined based solely on the extremities of the paths of $F_p$ and ignores
the internal vertices of those paths as well as all of $F_c$. So distinct factors can
yield the same route.

A k-digraph is called a {\em k-diregular digraph} ($k$-dd) if all of its vertices
are saturated; it is sometimes convenient in this case to say that $G$ is
{\em saturated}.
Clearly, a component of a k-digraph must also be a k-digraph (and {\em may} be a 2-dd);
likewise, a component of a k-dd must also be a k-dd.

We use $\partialG$ and $\fullG$ to denote the families of 2-digraphs
and 2-dds respectively. Clearly we have $\fullG \subset \partialG$.

We now generalize the notion of alternating cycles, defined for $\fullG$
in [\ref{ref-ram}], to $\partialG$.

Let $G = (V,A) \in \partialG$. A sequence of $2r$ {\em distinct} arcs
$X = (e_0, e_1, ..., e_{2r-1})$ in $A$ is an {\it alternating cycle (AC)}
if $e_i$ and $e_{i \oplus 1}$ have a common end-vertex [start-vertex] if $i$ is even [odd],
where $\oplus$ denotes addition mod $2r$. The arcs of an alternating cycle can be
partitioned into two disjoint sets called the set of {\em forward} and {\em backward}
arcs:
$X_f = \{ e_i \; | \; i \; \mbox{is even}\}$ and
$X_b = \{ e_i \; | \; i \; \mbox{is odd}\}$.
In [\ref{ref-ram}] we use the terms {\em clockwise} and {\em anti-clockwise} for these
arcs. The naming of the sets as ``forward'' and
``backward'' is arbitrary since traversing the arcs in the opposite direction switches
the sets. $X$ is called {\it even} or {\it odd} according as $r$ is even or odd. Clearly,
we must have $|X_f| = |X_b| = r$.

Informally, an AC is formed by starting with an arc, traversing it
forward, then traversing the next arc backward (which is guaranteed to be uniquely
possible since the end-vertex has indegree 2); this process of alternating forward
and backward traversals is continued until we return to the starting arc. So, an AC
always has an even number of arcs, half of which are forward and the rest backward.

Clearly, the arcs of $G$ can be uniquely partitioned into ACs in linear time. We show
later in proposition \ref{prop-factor} that a factor of $G$ can equivalently be defined
as a subgraph that includes all of $X_f$ and none of $X_b$ or vice versa for every AC
$X$ in $G$. A cycle [path] in $G$ is called a {\em difactorial cycle}
[{\em difactorial path}] if it is part of some factor of $G$; for any AC $X$, such a cycle
[path] cannot intersect both $X_f$ and $X_b$ (it may intersect neither).

We will henceforth represent $G \in \partialG$ by the triple $(V, A, C)$ where $C$
is the set of ACs.
The {\em index} of $G$, denoted by $i(G)$ is the smallest index among its factors, i.e.
\[
i(G) = \text{min}\{ i(F) \; | \; F \; \text{is a factor of} \; G\}
\]
Clearly, $G$ is open iff $i(G) = 0$; if $G \in \fullG$, it is Hamiltonian iff $i(G) = 1$.

Let $K \subseteq C, K \neq \emptyset$ and $G^K = (V^K, A^K, K)$ be the subgraph induced
by $K$. Clearly, $G^K \in \partialG$ and so $V^K$ is partitioned into sets
$V^K_{entry}, V^K_{exit}, V^K_{sat}$.
We will use $K$ and $G^K$ interchangeably
when there is little chance of confusion. The complement $\overline{K} = C-K$ also
induces a subgraph $G^{\overline{K}}$ or simply $\overline{K}$. When $K = \{X\}$ is a
singleton, we use $K$ and $X$ interchangeably. $G^K$ is {\em proper subgraph} if $K$ is a
proper subset of $C$. {\em In the rest of this document, the term} {\bf subgraph}
{\em will almost always mean such a 2-digraph induced by a subset of $C$}.

By {\em splitting} a saturated vertex $v$ we mean replacing it with 2 vertices, $v^{in}$
and $v^{out}$, where the two arcs of $v_{in}$ [$v_{out}$] become in-arcs [out-arcs] of
$v^{in}$ [$v^{out}$]. The inverse operation of {\em splicing} consists of identifying an
entry vertex $u$ with an exit vertex $v$ to create a new saturated vertex. Both
operations yield a 2-digraph.

If $G$ is closed, we say it is {\em minimally closed} if splitting $v$ makes $G$ open for
every $v \in V$,

If $G$ is closed, we can get a minimally closed graph $G'$ from it by splitting some
number of its saturated vertices.  We will explore such graphs later in section
\ref{sec-closed}.

A subset $S \subseteq V_{sat}$ is called a {\em split-set} if splitting all the vertices
of $S$ increases the number of components. $S$ is a {\em minimal split-set} if no
proper subset of $S$ is also a split-set. When $S$ is split, the components resulting
from the split are called {\em split-components}. $G$ is called {\em k-splittable}
if it has a split-set with $k$ elements and {\em minimally k-splittable} if that split
set is minimal.

For an AC $X$, an element of $X_{sat}$ is an {\em internal saturated vertex} or
simply an {\em internal vertex} of $X$. If $X$ has no internal vertices, it is called
{\em clean}, otherwise {\em dirty}. $G$ is {\em clean} if {\em all} of its ACs are clean,
{\em dirty} otherwise.

We can partition $V_{sat}$ into complementary subsets $V^i$ and $V^b$ defined as:
\[
V^i = \{v \; | \; v \; \mbox{is an internal vertex of some AC of $G$} \}
\]
and its complement $V^b = V_{sat} - V^i$. It is sometimes useful to refer to elements
of $G^b$ as {\em boundary vertices} since each is a member of two ACs.

For a factor $F$ of $G$, let $\overline{F}$ denote the subgraph induced by the
complementary arc-set $A-F$. Since $G \in \partialG$, clearly, $\overline{F}$ must
also be a factor of $G$. We call it the {\em complementary factor} or simply the
{\em complement}. Clearly if $G \in \fullG$, we must have $F_p = \emptyset$ for all
factors $F$.

A 2-dd $G \in \fullG$ is {\em odd} [{\em even}] iff $i(F)$ is odd [even] for every
factor $F$ (it is possible that $G$ is neither). We call a pair of 2-dds
{\em H-equivalent} iff they are either both Hamiltonian or both non-Hamiltonian.

For any family of 2-digraphs, we use the {\bf superscript} $2k$ to denote the sub-family
where each AC has exactly $2k$ arcs and {\em clean} [{\em dirty}] the sub-family of clean
[dirty] graphs.
Similarly, we use the {\bf subscript} $m$ to denote the sub-family where each graph has
exactly $m$ ACs and finally, {\em odd} [{\em even}] to denote the sub-family where every
AC is odd [even].

So, for notational illustration, we have the family hierarchies:
\[\partialG^{2k,clean}_{m} \subset \partialG^{2k}_m \subset \partialG^{2k} \subset \partialG
\] \[
\fullG^{2k,clean}_{m,odd} \subset \fullG^{2k,clean}_{m}
\subset \fullG^{2k}_m \subset \fullG^{2k} \subset \fullG
\] \[
\fullG^{clean}_{m, odd} \subset \partialG^{clean}_{m} \enspace \mbox{and} \enspace
\fullG^{clean}_{m, even} \subset \partialG^{clean}_{m}
\]

Some results below focus on the families $\partialG^6$ and $\fullG^6$, so it is useful to
consider all ten forms that an AC with six arcs can take (i.e all ten elements of
$\partialG^6_1$), shown in Figure \ref{fig-ac6} where the forward arcs are
solid and the backward, dotted. The last 3 graphs constitute $\fullG^6_1$.
Properties of these ten graphs are summarized in Table \ref{tab-ac}.

\begin{figure}[ht]
  \centering
%
\begin{tikzpicture}[->, node distance={10mm}, thick, main/.style = {circle}]    %
\tikzstyle{every text node part}=[font=\tiny, inner sep=.3]
\node[main] (a1) [draw] {$1$}; 
\node[main] (a2) [draw, above right of=a1] {$2$};
\node[main] (a3) [draw, right of=a2] {$3$}; 
\node[main] (a4) [draw, below right of=a3] {$4$};
\node[main] (a5) [draw, below left of=a4] {$5$}; 
\node[main] (a6) [draw, left of=a5] {$6$};
\node (name) at (1.3, 0) [font=\small] {$X_{clean}$};
\draw (a1) -> (a2);          
\draw[dotted] (a3) -> (a2);  
\draw (a3) -> (a4);
\draw[dotted] (a5) -> (a4);
\draw (a5) -> (a6);
\draw[dotted] (a1) -> (a6);          

\node[main] at (4,0) (b1) [draw] {$1$};                       
\node[main] (b2) [draw, above right of=b1] {$2$};
\node[main] (b3) [draw, right of=b2] {$3$}; 
\node[main] (b4) [draw, below=8mm of b3] {$4$};
\node[main] (b5) [draw, left of=b4] {$5$}; 
\node (name) at (5.1, 0) [font=\small] {$X_{1L}$};
\draw[dotted] (b1) to [out=160,in=200,looseness=12] (b1);
\draw (b1) -> (b2);
\draw[dotted] (b3) -> (b2);
\draw (b3) -> (b4);
\draw[dotted] (b5) -> (b4);
\draw (b5) -> (b1);

\node[main] at (7.3,0) (c1) [draw] {$1$};                       
\node[main] (c2) [draw, above right of=c1] {$2$};
\node[main] (c3) [draw, below right of=c1] {$3$}; 
\node[main] (c4) [draw, below left of=c1] {$4$};
\node[main] (c5) [draw, above left of=c1] {$5$}; 
\node (name) at (7.3, 0.9) [font=\small] {$X_{1S}$};
\draw (c1) to (c2);
\draw[dotted] (c3) -> (c2);
\draw (c3) -> (c1);
\draw[dotted] (c4) -> (c1);
\draw (c4) -> (c5);
\draw[dotted] (c1) -> (c5);

\node[main] at (9,0.5) (d1) [draw] {$1$};                     
\node[main] (d2) [draw, right of=d1] {$2$};                   
\node[main] (d3) [draw, below of=d2] {$3$}; 
\node[main] (d4) [draw, left of=d3] {$4$};
\node (name) at (9.5, 0) [font=\small] {$X_{2L}$};
\draw[dotted] (d1) to [out=110,in=70,looseness=12] (d1);
\draw[dotted] (d2) to [out=110,in=70,looseness=12] (d2);
\draw (d1) to (d2);
\draw (d2) -> (d3);
\draw[dotted] (d3) -> (d4);
\draw (d4) -> (d1);

\node[main] at (0,-3) (f1) [draw] {$1$};
\node[main] (f2) [draw, above right of=f1] {$2$};
\node[main] (f3) [draw, below right of=f1] {$3$}; 
\node[main] (f4) [draw, right=8mm of f1] {$4$};
\node (name) [font=\small, above of=f1] {$X_{2S}$};
\draw (f1) to (f2);
\draw[dotted] (f3) to [out=70,in=-70,looseness=1] (f2);
\draw (f3) to (f4);
\draw[dotted] (f2) -> (f4);
\draw (f2) to [out=250,in=110,looseness=1] (f3);
\draw[dotted] (f1) to (f3);

\node[main] at (2.4,-3) (g1) [draw] {$1$};
\node[main] (g2) [draw, above right of=g1] {$2$};             
\node[main] (g3) [draw, below right of=g1] {$3$};             
\node[main] (g4) [draw, right=8mm of g1] {$4$};
\node (name) [font=\small, above of=g1] {$X^c_{2L}$};
\draw (g1) to (g2);
\draw[dotted] (g2) to [out=70,in=110,looseness=11] (g2);
\draw (g2) to (g4);
\draw[dotted] (g3) -> (g4);
\draw (g3) to [out=250,in=290,looseness=11] (g3);
\draw[dotted] (g1) to (g3);

\node[main] at (5,-3) (h1) [draw] {$1$};                    
\node[main] (h2) [draw, right of=h1] {$2$};
\node[main] (h3) [draw, below right of=h2] {$3$};
\node[main] (h4) [draw, above right of=h2] {$4$};
\node (name) [font=\small, above of=h1] {$X^c_{1L1S}$};
\draw[dotted] (h1) to [out=160,in=200,looseness=11] (h1);
\draw (h1) to [out=20,in=160,looseness=1] (h2);
\draw[dotted] (h3) to (h2);
\draw (h3) -> (h4);
\draw[dotted] (h2) -> (h4);
\draw (h2) to [out=200,in=-20,looseness=1] (h1);


\node[main] at (0,-6) (i1) [draw] {$1$};              
\node[main] (i2) [draw, right of=i1] {$2$};
\node[main] (i3) [draw, right of=i2] {$3$};           
\node (name) [font=\small, above=3mm of i1] {$X^c_{2L1S}$};
\draw[dotted] (i1) to [out=160,in=200,looseness=11] (i1);
\draw (i1) to [out=20,in=160,looseness=1] (i2);
\draw[dotted] (i3) to [out=200,in=-20,looseness=1] (i2);
\draw (i3) to [out=20,in=-20,looseness=11] (i3);
\draw[dotted] (i2) to [out=20,in=160,looseness=1] (i3);
\draw (i2) to [out=200,in=-20,looseness=1] (i1);

\node[main] at (4.5,-6) (j1) [draw] {$1$};             
\node[main] (j2) [draw, right=8mm of j1] {$2$};      
\node[main] (j3) [draw, below left of=j2] {$3$};     
\node (name) [font=\small, above=8mm of j3] {$X^c_{3L}$};
\draw[dotted] (j1) to [out=160,in=200,looseness=11] (j1);
\draw (j1) to (j2);
\draw[dotted] (j2) to [out=20,in=-20,looseness=11] (j2);
\draw (j2) to (j3);
\draw[dotted] (j3) to [out=250,in=290,looseness=11] (j3);
\draw (j3) to (j1);

\node[main] at (8,-6) (j1) [draw] {$1$};
\node[main] (j2) [draw, right=8mm of j1] {$2$};
\node[main] (j3) [draw, below left of=j2] {$3$};
\node (name) [font=\small, above=8mm of j3] {$X^c_{3S}$};
\draw[dotted] (j2) to [out=150,in=30,looseness=1] (j1);
\draw (j1) to (j2);
\draw[dotted] (j1) to [out=-90,in=180,looseness=1] (j3);
\draw (j2) to (j3);
\draw[dotted] (j3) to [out=0,in=-90,looseness=1] (j2);
\draw (j3) to (j1);

\end{tikzpicture}
  \caption{Types of an AC with six arcs}
  \label{fig-ac6}
\end{figure}
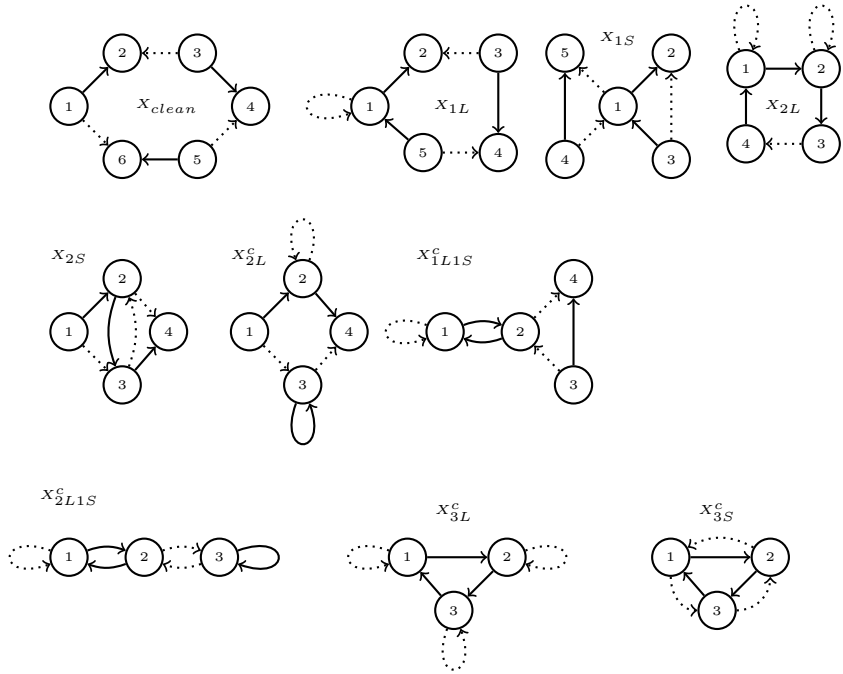

\begin{table}[ht]
  \caption{Properties of an AC with six arcs}
  \centering
  \begin{tabular}{l|l|l|l|l|l}
    \hline\hline
    &          &       & exit/entry & open factors, \\
    name & vertices & loops & vertices   & routes        & open/closed \\
    \hline
    $X_{clean}$   & 6 & 0 & 3            & 2, 2          & open \\
    $X_{1L}$      & 5 & 1 & 2            & 1, 1          & open \\
    $X_{1S}$      & 5 & 0 & 2            & 2, 1          & open \\
    $X_{2L}$      & 4 & 2 & 1            & 1, 1          & open \\
    $X_{2S}$      & 4 & 0 & 1            & 2, 1          & open \\
    \hline
    $X^c_{2L}$    & 4 & 2 & 1            & 0, 0          & closed \\
    $X^c_{1L1S}$  & 4 & 1 & 1            & 0, 0          & closed \\
    \hline
    $X^c_{2L1S}$  & 3 & 2 & 0            & 0, 0          & closed \\
    $X^c_{3L}$    & 3 & 3 & 0            & 0, 0          & closed \\
    $X^c_{3S}$    & 3 & 0 & 0            & 0, 0          & closed \\
    \hline
  \end{tabular}
  \label{tab-ac}
\end{table}

\section{Preliminaries}\label{sec-prelim}

We can show that connected graphs in $\fullG$ must be strongly connected. This is
true more generally for $k$-regular digraphs; to prove it we need the following lemma.

\begin{lemma}\label{lem-strong}
  Suppose $G = (V, A)$ is a $k$-regular digraph and $U \subset V$. Then,
  $|U_{in}| = |U_{out}|$.
\end{lemma}
\begin{proof} The proof is by induction on $|U|$. Since $G$ is $k$-regular, the result
  clearly holds when $|U| = 1$. Suppose it holds for all subsets of size $n$ and let
  $U$ be a subset of size $n+1$. For an arbitrary element $v \in U$ consider the set
  $W = U - \{v\}$. By the inductive hypothesis, $|W_{in}| = |W_{out}| = d$. Let $c_1$
  [$c_2$] be the count of arcs entering [exiting] $v$ from [to] $W$. The count of arcs
  entering [exiting] $v$ from [to] $V - W$ is $k - c_1$ [$k - c_2$]. Now
  $|U_{in}| = (k - c_1) + (d - c_2)$ and $|U_{out}| = (k - c_2) + (d - c_1)$ which
  completes the proof.
\end{proof}

\begin{theorem}
  Suppose $G = (V, A)$ is a connected $k$-regular digraph. $G$ must be strongly
  connected.
\end{theorem}
\begin{proof}
  Assume that $G$ is connected but not strongly connected. We can partition G into its
  stongly connected components $\{C_i\}_1^n$ ($n > 0$) and define the
  {\em superstructure} digraph $\tilde{G}$ in which there is an arc $(C_i, C_j)$ iff $G$
  has an arc $(u, v)$ with $u \in C_i, v\in C_j$ as described in [\ref{ref-even}]. This
  graph must be a connected DAG and hence must have a sink, say $C_m$. Since $C_m$ must
  have at least one arc entering from outside but cannot have any arcs exiting to the
  outside this is a contradiction of Lemma \ref{lem-strong} and the proof is concluded.
\end{proof}

\begin{prop}
  For $G \in \partialG$, we must have $|V_{entry}| = |V_{exit}|$ so $G$ must
  have an even number of unsaturated vertices.
  \label{prop-ee}
\end{prop}
\begin{proof}
  Let $G = (V,A,C)$ and consider the case when $V_{sat}$ is empty. The vertices of
  $X \in C$ must be disjoint from those of any other AC $Y \in C$ (otherwise we
  would have a saturated vertex), so $G$ is a disjoint collection of clean ACs.
  The number of exit and entry vertices of any single clean AC
  $(e_0, e_1, ..., e_{2r-1})$ are both $r$ so we have equality.

  If $V_{sat}$ is non-empty, by successively splitting each saturated vertex, we arrive
  at the previous case; we can then successively reverse the process by splicing the
  pairs of vertices that were previously split, preserving equality at each step until
  we arrive back at $G$.
\end{proof}

Let $X = (e_0, e_1, ..., e_{2r-1})$ be an AC of a 2-digraph $G = (V,A,C)$.
The following remarks are easy to see (some are stated in [\ref{ref-ram}] 
for the more restricted context of 2-dds but similar arguments apply here):

{\em REMARKS:}\\
\begin{enumerate}[topsep=0pt]
\item $|X_{entry}| = |X_{exit}|$
\item $|A| = |V| + |V_{sat}|$
\item We must have $r \le |V^X| \le 2r$. If $X$ is clean, we must have $|V^X| = 2r$
  and $|X_{entry}| = |X_{exit}| = r$.
\item If $X_{entry}$ and $X_{exit}$ are both empty, $r = |X_{sat}|$ and $X$ must be a
  component of $G$ and a 2-dd.
\item If $X$ and $Y$ are distinct ACs of $G$, we must have
  $X^i \cap Y^i = \emptyset$.
\setcounter{enumTemp}{\theenumi}
\end{enumerate}

The AC-decomposition is important because it is closely related to factors as shown by
the following proposition:
\begin{prop}\label{prop-factor}
  Consider the subgraph $F$ of $G$ obtained by arbitrarily choosing $X_f$
  or $X_b$ from each $X \in C$; $F$ must be a factor of $G$. Conversely, any factor
  $H = (V, B)$ of $G$ must have $B \cap X_f = \emptyset,  X_b \subseteq B$ or
  $B \cap X_b = \emptyset,  X_f \subseteq B$ for each $X \in C$.
\end{prop}
\begin{proof}
  Assume wlog that $F = (W, E)$ is a subgraph of $G$ obtained by selecting $X_f$ from
  every AC $X \in C$. Every vertex of $G$ is part of at least one AC $X$ and must be the
  start-vertex or end-vertex of an arc of $X_f$; hence, it must part of $W$. So $W = V$
  and $F$ is a spanning subgraph of $G$. A vertex $v$ of $G$ must be an element of
  $V_{entry}$, $V_{exit}$ or $V_{sat} $; in each case, we will have
  $|v_{out}| = 1, |v_{in}| = 0$, $|v_{in}| = 1, |v_{out}| = 0$, or $|v_{in}| = 1 = |v_{out}|$
  respectively which makes $F$ a 1-digraph and therefore a factor.

  For the converse, consider an arbitrary $u \in V_{entry}$. Let $e_0 = (u, v)$ be one of
  the two outgoing arcs of $u$ and $X = (e_0, e_1, ..., e_{2r-1})$ be the AC containing
  $e_0$. Since $F$ is a factor, it must include exactly one element of $u_{out}$. Assume
  wlog that it includes $e_0$. It cannot include $e_1 = (w, v)$ (since it must have
  exactly one of $v_{in}$); it must therefore include $e_2$ (again, since it must
  have exactly one element of $w_{out}$). The result follows by continuing this argument
  until we return to $e_0$.
\end{proof}

{\bf Corollaries}:
\begin{enumerate}[(a)]
\item $G$ has $2^{|C|}$ factors or $2^{|C|-1}$ complementary factor pairs.
\item For any non-empty $K \subseteq C, |K| = k$, the subgraph $G^K = (V^K, A^K, K)$
  has $2^k$ factors and each factor $F$ defines a route $r_F$ mapping $V^K_{entry}$ to
  $V^K_{exit}$.
\end{enumerate}

\section{Split Sets}\label{sec-split}
We now look at how split-sets interact with factors. The following remarks are easy to
see ($G \in \partialG$):\\
{\em REMARKS:}\\
\begin{enumerate}[topsep=0pt]
\setcounter{enumi}{\theenumTemp}
\item If $G$ is connected and $S = \{v\}$ is a singleton split-set, it must
  have exactly two split-components, one containing $v_{in}$ and the other $v_{out}$.
\item If $S$ is a minimal split-set it must be entirely in one component of $G$.
\item If $G$ is connected and $S$ is a split-set with split-components
  $H^j, j = 1, ..., n$, we must have:
  \begin{gather*}
  |V_{sat}| = \sum_1^n|H^j_{sat}| + |S| \\
  |V_{entry}| = \sum_1^n|H^j_{entry}| - |S| \\
  |V_{exit}| = \sum_1^n|H^j_{exit}| - |S|
  \end{gather*}
    {\it Proof}: The saturated vertices of $H^j$ must also be saturated in $G$;
    a saturated vertex of $G$ on the other hand is also one in one of the $H^j$ iff
    it was not split. Similarly, an entry [exit] vertex of $G$ must also be one in
    (exactly) one $H^j$; however, an entry [exit] vertex of $H^j$ is either
    an entry [exit] vertex of $G$ or was an artifact of the split.
\setcounter{enumTemp}{\theenumi}
\end{enumerate}

When $S$ is a minimal split set of a connected 2-dd, we can say more:
\begin{prop}\label{prop-even-split-set}
  Suppose $G \in \fullG$ is connected and $S$ is a minimal split set.
  \begin{enumerate}[(a)]
  \item \label{two-split} Splitting $S$ yields exactly two split components.
  \item \label{diff-components} For all $v \in S$,
    $v^{in}$ and $v^{out}$ must be in different split components.
  \item \label{no-internal} $S \cap G^i = \emptyset$
  \item \label{even-split} $|S|$ is even.
  \end{enumerate}
\end{prop}
\begin{proof}
  Consider the situation when $S' = S - \{v\}$ is split; since $S$ is minimal, we must
  still have a connected graph. Now splitting $v$ can only yield two components, so
  \ref{two-split} holds.

  For \ref{diff-components}, if $v^{in}$ and $v^{out}$ were in the
  same component, splicing them back would still leave us with two components violating
  the minimality of $S$, so \ref{diff-components} holds.

  If $v \in S \cap G^i$, splitting it requires both new vertices must be in the same
  component since they are part of the same AC which violates the previous item, so
  \ref{no-internal} follows from \ref{diff-components}.

  Finally, for \ref{even-split}, suppose $|S| = 2m + 1$ for some $m > 0$ and let
  $H_1$ and $H_2$ be the two split
  components resulting from splitting $S$. Consider again the situation when $S'$ is
  split: Since $2m$ vertices have been split, yielding $4m$ unsaturated vertices, we
  must (by (b)) have $2m$ in each of $H_1$ and $H_2$. When $v$ is split, each component
  has $2m + 1$ unsaturated vertices which contradicts proposition \ref{prop-ee}.
\end{proof}

Recall that, in a 2-digraph that is not saturated, each factor defines a route and that
distinct factors may define the same route. Since routes are bijections, they can be
viewed as permutations in $S_n$. We can use Theorem 2 of [\ref{ref-ram}] to deduce some
properties of these permutations:
\begin{theorem}\label{thm-parity}
  Suppose $G = (V, A, C) \in \partialG_{odd}$ with $|V_{exit}| = n > 0$.
  Let $R_{even}$ [$R_{odd}$] be the set of routes defined by all factors of $G$ with
  even [odd] indices.\\
  (a) $R_{odd}$ and $R_{even}$ are disjoint.\\
  (b) By labeling $V_{entry}$ and $V_{exit}$ with $[1, n]$, each route $r \in R_{odd}$ can
  be viewed as a permutation in $S_n$. All such permutations must have the same
  parity. The same holds for permutations defined by the routes of $R_{even}$.
\end{theorem}
\begin{proof}
  Using an arbitrary bijection $f$ from $V_{entry}$ to $V_{exit}$, pair each 
  $u \in V_{entry}$ with $v = f(u) \in V_{exit}$ and consider the 2-dd $G'$ obtained by
  splicing all the $(u, v)$ pairs. A factor $F$ of $G$ becomes a factor $F'$ of $G'$
  with $F'_p = \emptyset$; the paths of $F_p$ become a collection $D_F$ of disjoint
  cycles in $F'_c$, so $i(F') = i(F) + |D_F|$. Similarly, if $H$ is any other factor of
  $G$, $i(H') = i(H) + |D_H|$.

  To prove (a), observe that the only way a route $r_F$ can be in both $R_{odd}$ and
  $R_{even}$ is if we have distinct factors $F$ and $H$ with $i(F)$ and $i(H)$ having
  different parities but defining the same route, i.e. $r_F = r_H$. This means
  $|D_F| = |D_H|$. By Theorem 2 of [\ref{ref-ram}], $i(F')$ and $i(H')$ must have the
  same parity, and therefore so must $i(F)$ and $i(H)$ and (a) is proved.

  To prove (b), assume $i(F)$ and $i(H)$ are odd. A labelling of $V_{entry}$ and $V_{exit}$
  with $[1, n]$ yields a natural pairing $(i, i)$ of entry vertices to exit vertices,
  which means that $D_F$ and $D_H$ are cyclic decompositions of permutations $p_F$ and
  $p_H$ in $S_n$ defined by the corresponding routes $r_F, r_H \in R_{odd}$. Again, by
  Theorem 2 of [\ref{ref-ram}], $i(F')$ and $i(H')$ must have the same parity, so $|D_F|$
  and $|D_H|$ have the same parity, so the parity of the corresponding permutations
  $p_F$ and $p_H$ must be the same. The proof for $R_{even}$ is similar.
\end{proof}

{\bf Corollaries}:
\begin{enumerate}[(a)]
\item
  All permutations defined by open routes of $G$ have the same parity.
  {\em Proof}: An open route is defined by an open factor; an open factor has zero index
  since it has no cycles.
\item
  $|V_{exit}| = n = 2$ and $G$ is open $\implies G$ has a unique open route.
  {\em Proof}: By Proposition \ref{prop-ee}, $G$ must have two entry vertices, so let
  $V_{entry} = \{u_0, u_1\}, V_{exit} = \{v_0, v_1\}$ and $r$ is an open route that maps
  $u_i$  to $v_i$. Let $r' \neq r$ be another open route; it must map $u_0$ to $v_1$ and
  $u_1$ to $v_0$. Clearly, the permutations $p_0, p_1 \in S_2$ defined by these routes
  will have different parities (one is the identity and the other a transposition) which
  contradicts the corollary above.
\end{enumerate}
\begin{theorem} \label{thm-split2}
  Suppose $G \in \fullG_{odd}$ is odd and $S = \{u, v\}$ a minimal split-set with split
  components $G_1$ and $G_2$. Assume wlog that they contain $\{u^{in}, v^{out}\}$ and
  $\{u^{out}, v^{in}\}$ respectively. Let $G'_1, G'_2 \in \fullG$ be the 2-dds obtained by
  splicing these two pairs. Then,\\
  (a) $G$ is Hamiltonian $\iff$ both $G'_1$ or $G'_2$ are Hamiltonian.\\
  (b) $G'_1$ or $G'_2$ is even $\implies$ both are even (and $G$ is non-Hamiltonian).
\end{theorem}
\begin{proof}
  A Hamiltonian cycle of $G$ splits into a Hamiltonian path in each of the two split
  components; both are turned into Hamiltonian cycles in their respective components by
  the splicing so the forward implication of (a) is established. The converse is also
  similarly obvious. For (b), assume wlog that $G'_1$ is even. If $G'_2$ is odd, let
  $F_1$ and $F_2$ be arbitrary factors of $G'_1$ and $G'_2$ respectively with
  $i(F_1) = 2m$ and $i(F_2) = 2n+1$. Now, consider the unique cycles $L_1$ and $L_2$ in
  the respective factors which contain the spliced vertices in the two components. These
  must be the two parts of a single cycle $L$ in $G$, so the number of components in the
  corresponding factor of $G$ must be $(2m - 1) + (2n + 1 - 1) + 1 = 2(m+n)$ which is
  even. This contradicts our assumption that $G$ is odd.
\end{proof}

This theorem also yields a simple, computationally efficient, sufficient condition for
non-Hamiltonicity of odd 2-splittable graphs in $\fullG_{odd}$: given such a graph
$G = (V, A, C)$, create the split components and splice the exit and entry vertices of
each to get two new 2-dds $G_1, G_2 \in \fullG_{odd}$. Now create a factor $F$ in either
by arbitrarily choosing either $X_f$ or $X_b$ from each AC $X$; if $F$ is even, $G$ is
not Hamiltonian; otherwise, it is possible that one or both are themselves 2-splittable;
if so, repeat the process. When this iterative process concludes, either we have
established the non-Hamiltonianicity of $G$ or we have a set of graphs
$\{G_i = (V_i, A_i, C_i)\}_1^n$ in $\fullG_{odd}$ ($n \leq |C|$, $\sum_i |C_i| =  |C|$)
each of which is odd and non-splittable with the property that if any of its elements is non-Hamiltonian, so is $G$.

The connected 2-dd of Figure \ref{fig-ac4-nonH} is in $\fullG^{6,clean}_4$ is odd,
strongly connected, minimally 2-splittable with split-set $S = \{5, 7\}$; by splitting
$S$ and splicing the unsaturated vertices in each component as described above, we can
verify that both of the resulting 2-dds are even and conclude that the original graph is
non-Hamiltonian.
\begin{figure}[!h]
  \centering
%
\begin{tikzpicture}[->, node distance={10mm}, thick, main/.style = {circle}]    %
\tikzstyle{every text node part}=[font=\tiny, inner sep=.3]
\tikzset{node distance = 2cm and 2cm}

\node[main] (a6) [draw] {$6$}; 
\node[main] (a1) [draw, below of=a6] {$1$};
\node[main] (a2) [draw, below of=a1] {$2$}; 
\node[main] (a3) [draw, right of=a1] {$3$};
\node[main] (a4) [draw, right of=a3] {$4$}; 
\node[main] (a5) [draw, above right of=a4] {$5$};
\node[main] (a7) [draw, below right of=a4] {$7$};

\node[main] (a8)  [draw, below right of=a5] {$8$};
\node[main] (a10) [draw, right of=a8] {$10$};
\node[main] (a9)  [draw, above of=a10] {$9$};
\node[main] (a11) [draw, below of=a10] {$11$};
\node[main] (a12) [draw, right of=a10] {$12$};

\draw[red] (a6) -> (a3);
\draw[red, dotted] (a4) to [out=225,in=-45,looseness=1.5] (a3);
\draw[red] (a4) -> (a7);
\draw[red,dotted] (a2) -> (a7);
\draw[red] (a2) to (a1);
\draw[red,dotted] (a6) to [out=225,in=135,looseness=1.5] (a1);

\draw[blue] (a1) to [out=225,in=135,looseness=1.5] (a2);
\draw[blue, dotted] (a3) to  (a2);
\draw[blue] (a3) to [out=45,in=135,looseness=1.5] (a4);
\draw[blue,dotted] (a5) -> (a4);
\draw[blue] (a5) to (a6);
\draw[blue,dotted] (a1) to (a6);

\draw[green] (a7) to [out=-60,in=-90,looseness=1.8] (a12);
\draw[green, dotted] (a9) to (a12);
\draw[green] (a9) to [out=-60,in=45,looseness=1.5] (a10);
\draw[green,dotted] (a11) to [out=120,in=225,looseness=1.5] (a10);
\draw[green] (a11) to (a8);
\draw[green,dotted] (a7) to (a8);

\draw[magenta] (a8) to (a5);
\draw[magenta, dotted] (a12) to [out=90,in=45,looseness=1.8] (a5);
\draw[magenta] (a12) to (a11);
\draw[magenta,dotted] (a10) to (a11);
\draw[magenta] (a10) to (a9);
\draw[magenta,dotted] (a8) to (a9);

\end{tikzpicture}
  \caption{Strongly connected, odd, non-Hamiltonian 2-dd in $\fullG^{6,clean}_4$
    with minimal split set $\{5, 7\}$}
  \label{fig-ac4-nonH}
\end{figure}
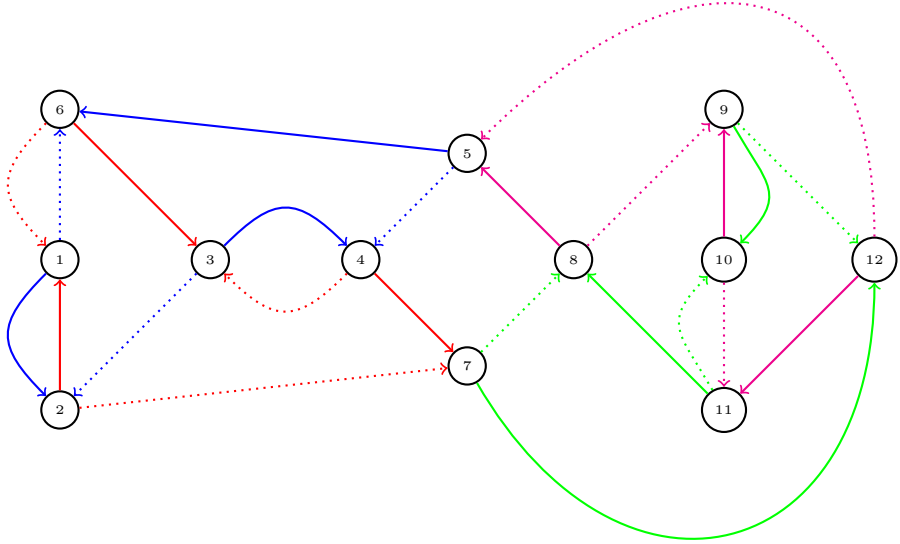

The non-Hamiltonicity of Figure 4(b) of [\ref{ref-ram}] can also be established by a
very similar argument.

For another illustration of the utility of Theorem \ref{thm-split2}, consider the family
$B^6_6$, which has, per the results of [\ref{ref-ramwalsh}, \ref{ref-ramwalsh-tech}]
218,161,485 graphs of which 212,785,467 are (strongly) connected; of these, 12,513 are
clean, odd and non-Hamiltonian. The non-Hamiltonicity of 11,320 of these
($\sim$90\%) can be established by the above Theorem: they are 2-splittable with both
split components even after the requisite splicing; only 1,193 are not amenable to this
criterion.

We can reverse this process to get an infinite sub-family of $\fullG_{odd}$ consisting of
strongly connected, odd, non-Hamiltonian 2-dds as follows: choose an arbitrary pair of
connected even graphs from $\fullG_{odd}$, split a single vertex in each and splice them
together in the obvious way; all such resulting graphs must be odd and non-Hamiltonian.

Though individual non-Hamiltonian regular connected graphs of low degree are known
(e.g. Tutte and Meredith graphs [\ref{ref-tutte}, \ref{ref-meredith}]) infinite
families of such graphs are not common in the literature; even scarcer are infinite
families of such digraphs.

As follows from Theorem 2 of [\ref{ref-ram}], the connected elements of $\fullG_{odd}$
can be partitioned into a pair of classes $\family{F}_1$ and $\family{F}_2$
of odd and even graphs; the latter is an infinite family of strongly connected
non-Hamiltonian 2-dds.

The result of [\ref{ref-trotter}] yields yet another such family consisting of 2-dds
that are products of directed cycles.

\section{Minors and Quotients}\label{sec-minors}
Suppose $G = (V, A, C) \in \partialG$, $K \subset C, K \neq \emptyset$. For any
{\bf open} route $r$ of $K$, we
define ${r \downarrow}$, the {\em minor defined by $r$}, as the minor obtained from $G$
as follows (let $F$ be any factor of $K$ with $r = r_F$):
\begin{enumerate}
\item Delete all arcs and saturated vertices of $K$.
  \label{minor-del-all}
\item Identify each pair $(u,v) \in r$ to form a new vertex $w$ (any in-arcs of $u$
  and out-arcs of $v$ in $\overline{K}$ are preserved as incident arcs of $w$).
  \label{minor-splice}
\item Finally, delete any isolated vertices in the resulting graph (cannot occur
  if $G \in \fullG$).
  \label{minor-del-isolated}
\end{enumerate}

\begin{prop}
  The following hold:
  \begin{enumerate}[topsep=0pt]
  \item ${r \downarrow} \in \partialG$. \label{prop-reduction-membership}
  \item $G \in \fullG \implies {r \downarrow} \in \fullG$.
     \label{prop-reduction-membershipF}
  \item ${r \downarrow}$ has $2^{|\overline{K}|}$ factors. \label{prop-reduction-factors}
  \end{enumerate}
\end{prop}
\begin{proof}
  Since the AC set of ${r \downarrow}$ is $\overline{K}$,
  \ref{prop-reduction-factors} is obvious from proposition \ref{prop-factor}.
  For \ref{prop-reduction-membership}, vertices not in $K$ are unaffected by the
  construction of the minor so we only need to show that the new vertex $w$ created in
  step \ref{minor-splice} satisfies the degree constraints of a 2-digraph. We
  distinguish four cases for the identified vertices $u$ and $v$:
  \begin{itemize}[topsep=0pt]
  \item Case A: $u$ and $v$ are also entry and exit vertices respectively of $G$.
    In this case, $w$ becomes an isolated vertex deleted in step
    \ref{minor-del-isolated}.
  \item Case B: $u$ and $v$ are both saturated vertices of $G$. Here,
    $w_{in} = w_{out} = 2$ in ${F\downarrow}$ since the two incoming arcs of $u$ and
    the two outgoing arcs of $v$ are preserved as incident arcs of $w$.
  \item Case C: $u$ is an entry vertex but $v$ is a saturated vertex of $G$.
    Here, $w_{in} = 0, w_{out} = 2$ in ${F\downarrow}$ since the two outgoing arcs of
    $v$ are preserved as incident arcs of $w$.
  \item Case D: $u$ is a saturated vertex but $v$ is an exit vertex of $G$. Follows
    by symmetry with the previous case.
  \end{itemize}
  This completes the proof of \ref{prop-reduction-membership}.
  If $G \in \fullG$, only case B is possible since all vertices of $G$ are
  saturated, so $w$ satifies the necessary degree constraint for ${r \downarrow}$
  to be in $\fullG$, so \ref{prop-reduction-membershipF} holds.
\end{proof}

The set of all such minors defined by all open routes of $K$ is called the
{\em K-quotient of $G$} and is denoted by:
\[
G/K = \{ {r \downarrow} \; | \; r \; \text{is an open route of} \; G^K\}
\]
If $K = \{X\}$, we use $G/X$ to mean $G/K$. We say that
$G/K$ is Hamiltonian iff it contains a Hamiltonian element.

The following remarks should now be obvious: \\
{\sl REMARKS:}\\
\begin{enumerate}[topsep=0pt]
\setcounter{enumi}{\theenumTemp}
\item $0 \leq |G/K| \leq 2^{|K|}$ (if many factors define the same route, $|G/K|$ will be
  considerably smaller than $2^{|K|}$)
\item $G/K$ may contain isomorphic copies.
\item \label{closed-ac} $G/K$ is empty iff $K$ is closed.
\setcounter{enumTemp}{\theenumi}
\end{enumerate}

Consider a minor $M \in G/K, M = {r \downarrow}$ for some open route $r$ of $K$ and $F$
a factor of $K$ with $r_F = r$. A factor $J$ of $M$ can be mapped to a factor $J'$ of $G$
as follows: Select the same arcs as $F$ from $K$ and the same arcs from $\overline{K}$ as
$J$ (this is possible since the AC-set of $M$ is $\overline{K}$). It is easy to see that
$J'$ is a factor of $G$ and $i(J) = i(J')$.

Conversely, given a factor $L$ of $G$ which selects the same arcs from $K$ as $F$, it
can be mapped to a factor $L'$ of $M$ obtained by selecting the same arcs from
$\overline{K}$ as $L$. Here again it is obvious that $i(L) = i(L')$.
We have just proved:

\begin{prop} \label{prop-reduction}
  Suppose $K \subset C, K \neq \emptyset$, $F$ an open factor that defines the open
  route $r$ of $K$. Let $P$ be the
  set of factors of ${r \downarrow{}}$ and $Q$ the set of factors of $G$ that coincide
  with $F$ in $K$. Then, there is a unique bijection: $f: P \rightarrow Q$ with $J$ and
  $f(J)$ coinciding in $\overline{K}$ and $c(J) = c(f(J))$, $i(J) = i(f(J))$.
\end{prop}

This proposition yields the following theorem as an easy consequence:
\begin{theorem} Suppose $G = (V, A, C) \in \fullG$.\\ \label{thm-reduction}
  \begin{enumerate}[(a),topsep=0pt]
    \item If $G$ is Hamiltonian, so is $G/K$ for every $K \subset C, K \neq \emptyset$.
      Conversely, if $G/K$ is Hamiltonian for {\bf some} $K \subset C, K \neq \emptyset$,
      $G$ Hamiltonian. \label{thm-reduction-a}
    \item If $K \subset C$ has a unique open route $r$, the minor ${r \downarrow{}}$ is
      Hamiltonian-equivalent to $G$. \label{thm-reduction-b}
    \item Any $G \in \fullG^6$ that contains $X^c_{L2}$ or $X^c_{L1S1}$ is non-Hamiltonian.
      \label{thm-reduction-c}
  \end{enumerate}
\end{theorem}
\begin{proof}
  If $H$ is a Hamiltonian cycle of $G$, its restriction $F$ to $K$ must be
  open; hence, $r \downarrow$ exists and $f^{-1}(H)$ is the corresponding Hamiltonian
  cycle in $r \downarrow$. Conversely, if $H$ is Hamiltonian cycle of $r \downarrow$,
  $f(H)$ is a Hamiltonian cycle in $G$ so \ref{thm-reduction-a} is proved.

  \ref{thm-reduction-b} follows from \ref{thm-reduction-a} since $G/K$ is singleton.

  \ref{thm-reduction-c} is obvious since $G/X$ is empty for each of these ACs by remark
  \ref{closed-ac}.
\end{proof}

We now prove a necessary condition for an odd element of $\fullG^6$ to be
non-Hamiltonian:
\begin{prop} \label{prop-non-ham}
  Suppose $G = (V, A, C) \in \fullG^6$ is connected, odd and non-Hamiltonian.
  Then, for any open AC $X \in C$, $\overline{X}$ is closed.
\end{prop}
\begin{proof}
  Let $X \in C$ be an arbitrary open AC with six arcs
  and assume that $\overline{X}$ is open. Clearly $X \notin \{X^c_{2L}, X^c_{1L1S}\}$
  since both of those are closed.\\
  {\bf Case A}: $X = X_{2S}$. Any open factor of $\overline{X}$ must be a single
  Hamiltonian path from the unique entry vertex to the unique exit vertex; by
  appending either of the two open paths formed by $X_f$ and $X_b$ to this path
  we get a Hamiltonian circuit of $G$ which is a contradiction.\\
  {\bf Case B}: $X = X_{2L}$. Similarly, any open factor of $\overline{X}$ must
  be a single Hamiltonian path from the unique entry vertex to the unique exit
  vertex; by appending the single open path in $X$ (i.e. the solid arcs) to this
  path we get a Hamiltonian circuit of $G$ which is a contradiction.\\
  {\bf Case C}: $X = X_{1S}$ Here, $\overline{X}$ has two entry vertices $\{5, 2\}$ and
  two exit vertices $\{4, 3\}$, so an open factor must consist of
  two simple paths $5 \rightsquigarrow 4$ and $2 \rightsquigarrow 3$ or vice versa.
  In one case, adding either $X_f$ or $X_b$ yields an even factor of $G$ and, in
  the other case, a Hamiltonian circuit both of which contradict our assumptions.\\
  {\bf Case D}: $X = X_{1L}$. This case is similar to the previous case.\\
  {\bf Case E}: $X = X_{clean}$. Here, $\overline{X}$ has three entry vertices
  $\{2, 4, 6\}$ and three exit vertices $\{1, 3, 5\}$, so an open factor must
  consist of three simple paths. In each of the six possible scenarios, adding
  either $X_f$ or $X_b$ will yield the desired contradiction. For example, paths
  $2 \rightsquigarrow 5$, $4 \rightsquigarrow 1$ and $6 \rightsquigarrow 3$ yield
  Hamiltonian circuits when combined with either $X_f$ or $X_b$.
\end{proof}

Theorem \ref{thm-reduction} and proposition \ref{prop-non-ham} above allow us to state
the following necessary and sufficient condition for non-Hamiltonicity of odd elements of
$\fullG^6$:
\begin{theorem} \label{thm-non-ham}
  Suppose $G = (V, A, C) \in \fullG^6$ is odd. $G$ is non-Hamiltonian iff for
  some $K \subset C, K \neq \emptyset$, $K$ is closed.
\end{theorem}

The following result applies Theorem \ref{thm-reduction} to $\fullG^6$ to
show that, in the context of Hamiltonicity, dirty ACs can be eliminated.

\begin{prop}
  Suppose $G = (V, A, C) \in \fullG^6$ and $|C| > 1$. If
  $X \in \{X_{1L}, X_{1S}, X_{2L}, X_{2S}\}$ is an AC of $G$,
  $G/X$ is a singleton and its sole member is Hamiltonian-equivalent to $G$.
\end{prop}
\begin{proof} All four ACs have a unique open route and so, in each case, $G/X$ is
  singleton; the result now follows from theorem \ref{thm-reduction}.
\end{proof}

In an algorithmic context, it is important to note that the unique minor in the quotient
may:
\begin{enumerate}[(a),topsep=0pt]
\item be disconnected.
\item contain a dirty AC that was clean in $G$.
\item contain $X^c_{L2}$ or $X^c_{L1S1}$.
\end{enumerate}
In the first case, we can immediately conclude that $G$ is non-Hamiltonian; likewise in
the last case using Theorem \ref{thm-reduction}\ref{thm-reduction-c}.
So, if $G$ is connected and dirty, we can repeatedly apply this proposition to eliminate
all dirty ACs and end up either with a connected clean graph $G'$ that is
Hamiltonian-equivalent to $G$ or conclude that $G$ is non-Hamiltonian.

\section{Closed graphs}\label{sec-closed}
In the light of Theorem \ref{thm-non-ham}, closed 2-digraphs are useful in
determining non-Hamiltonicity of 2-dds so in this section, we deduce some properties of
such graphs.

Let $\family{C}^6 \subset \partialG^6$ denote the sub-family all of whose graphs are
minimally closed, connected, and clean. We show that a substantial percentage of
the vertices of graphs in $C^6$ must be saturated.

Suppose $G = (V, A, C) \in \partialG$.

The following remarks should be obvious since every factor $F$ of $G$, when restricted
to a subgraph $H$ of $G$, yields a factor of $H$:\\
{\em REMARKS:}
\begin{enumerate}[topsep=0pt]
  \setcounter{enumi}{\theenumTemp}
  \item $G$ is open iff every subgraph of $G$ is open.
  \item \label{rem-closed-component} $G$ is closed iff at least one component of $G$ is
    closed.
  \item $i(G) = \sum_{H} i(H)$ where the sum is over all components $H$ of $G$.
  \setcounter{enumTemp}{\theenumi}
\end{enumerate}

\begin{prop} \label{lem-closed-connected}
  Suppose $G$ is connected and minimally closed. Then:
  \begin{enumerate}[(a),topsep=0pt]
  \item\label{prop-subgraph-open} Every proper subgraph of $G$ must be open.
  \item \label{prop-singleton-split} $G$ cannot have a singleton split set.
  \end{enumerate}
\end{prop}
\begin{proof}
  For \ref{prop-subgraph-open}, if a proper subgraph $K \subset C$ is closed, we can
  split a boundary vertex of K and still have a closed graph which contradicts the
  assumption that $G$ is minimally closed.

  For \ref{prop-singleton-split}, if $S = \{u\}$ is a singleton split set, splitting $u$
  yields split components $G_1, G_2$ containing respectively $u^{in}$ and $u^{out}$. Now, no
  simple cycle of $G$ can contain $u$ since, once it traverses $u$ from $G_1$ to $G_2$, it
  has no way to return to the starting vertex. By \ref{prop-subgraph-open}, neither $G_1$
  nor $G_2$ can be closed; but this means $G$ is open which is a contradiction.
\end{proof}

\begin{lemma} \label{lem-closed-saturated}
  Let $G \in \family{C}^6$, and $X$ an AC of $G$. We must have $|X^b| >= 4$.
\end{lemma}
\begin{proof} 
  Let $u \in G_{sat}$ be an entry vertex of $X$. Since splitting $u$ creates an open
  factor, we must have a factor $F$ where the cycle $D$ containing $u$ is the only
  cycle in it. Assume wlog that $D$ traverses the arc $(u,v) \in X_f$; using $X_{clean}$
  from Figure \ref{fig-ac6} for reference, assume that this is the arc $(1, 2)$.
  Consider the factor $F'$ obtained by replacing $X_f$ with $X_b$; $D$ becomes
  $D' = (3, 2, ... 1, 6)$. Since $F'$ must be closed and it cannot have a circuit
  disjoint from $X$, we have only two possibilities: either $D'$ is part of a difactorial
  circuit (meaning there is difactorial path from $6$ to $3$), or we have a circuit
  containing the arc $(5,4)$. In both cases we have the desired four saturated vertices
  in $X$ (either $\{1,2,3,6\}$ or $\{1,2,4,5\}$). A parallel argument applies in the case
  when $u$ is an exit vertex of $X$.
\end{proof}

\begin{theorem} \label{theorem-sat}
If $G \in \family{C}^6$ has $m$ ACs, $|V_{sat}| >= 2m$.
\end{theorem}
\begin{proof}
  From lemma \ref{lem-closed-saturated} above:
  \[ \sum_{X \in C}|X^b| \geq 4m \]
  Since each of AC of $G$ is clean, each $v \in V_{sat}$ is contained in a unique pair of
  ACs and hence is counted twice in this sum; the result follows.
\end{proof}

For an example of the use of Theorem \ref{thm-reduction}, consider the strongly
connected, 30-vertex 2-dd of Fig. 8 in [\ref{ref-delorme}] (that paper does not deal with
Hamiltonicity); it is redrawn in Figure \ref{fig-v30} below with the vertices labeled,
the arcs of the two ACs in red and blue using solid/dotted lines for the
forward/backward arcs. We can see that both ACs have 30 arcs, are dirty and closed, so
both quotients are empty. Hence the graph is non-Hamiltonian.
\begin{figure}[ht]
  \centering
%
\begin{tikzpicture}[->, node distance={10mm}, thick, main/.style = {circle}]    %
\tikzstyle{every text node part}=[font=\tiny, inner sep=.3]


\node[main] (a1) at (0:1) [draw] {$1$};
\node[main] (a5) at (72:1) [draw] {$2$};
\node[main] (a4) at (2*72:1) [draw] {$3$};
\node[main] (a3) at (3*72:1) [draw] {$4$};
\node[main] (a2) at (4*72:1) [draw] {$5$};
\node[main] (a6) at (-36:2) [draw] {$6$};
\node[main] (a7) at (36:2) [draw] {$7$};
\node[main] (a8) at (36+1*72:2) [draw] {$8$};
\node[main] (a9) at (36+2*72:2) [draw] {$9$};
\node[main] (a10) at (36+3*72:2) [draw] {$10$};

\node[main] (a14) at (18:3)    [draw] {$14$};
\node[main] (a15) at (18+36:3)   [draw] {$15$};
\node[main] (a16) at (18+2*36:3) [draw] {$16$};
\node[main] (a17) at (18+3*36:3) [draw] {$17$};
\node[main] (a18) at (18+4*36:3) [draw] {$18$};
\node[main] (a19) at (18+5*36:3) [draw] {$19$};
\node[main] (a20) at (18+6*36:3) [draw] {$20$};
\node[main] (a11) at (18+7*36:3) [draw] {$11$};
\node[main] (a12) at (18+8*36:3) [draw] {$12$};
\node[main] (a13) at (18+9*36:3) [draw] {$13$};

\node[main] (a28) at (0:4.5)    [draw] {$28$};
\node[main] (a23) at (36:4.5)   [draw] {$23$};
\node[main] (a29) at (2*36:4.5) [draw] {$29$};
\node[main] (a24) at (3*36:4.5) [draw] {$24$};
\node[main] (a30) at (4*36:4.5) [draw] {$30$};
\node[main] (a25) at (5*36:4.5) [draw] {$25$};
\node[main] (a26) at (6*36:4.5) [draw] {$26$};
\node[main] (a21) at (7*36:4.5) [draw] {$21$};
\node[main] (a27) at (8*36:4.5) [draw] {$27$};
\node[main] (a22) at (9*36:4.5) [draw] {$22$};

\draw[red] (a1) -> (a5);
\draw[red] (a5) -> (a4);
\draw[red] (a4) -> (a3);
\draw[red] (a3) -> (a2);
\draw[red] (a2) -> (a1);
\draw[red,dotted] (a1) -> (a6);
\draw[red,dotted] (a6) -> (a2);
\draw[red,dotted] (a2) -> (a10);
\draw[red,dotted] (a10) -> (a3);
\draw[red,dotted] (a3) -> (a9);
\draw[red,dotted] (a9) -> (a4);
\draw[red,dotted] (a4) -> (a8);
\draw[red,dotted] (a8) -> (a5);
\draw[red,dotted] (a5) -> (a7);
\draw[red,dotted] (a7) -> (a1);

\draw[blue,dotted] (a11) -> (a20);
\draw[blue,dotted] (a19) -> (a18);
\draw[blue,dotted] (a17) -> (a16);
\draw[blue,dotted] (a15) -> (a14);
\draw[blue,dotted] (a13) -> (a12);
\draw[red,dotted] (a20) -> (a19);
\draw[red,dotted] (a18) -> (a17);
\draw[red,dotted] (a16) -> (a15);
\draw[red,dotted] (a14) -> (a13);
\draw[red,dotted] (a12) -> (a11);
\draw[red] (a12) -> (a6); \draw[red] (a6) -> (a13);
\draw[red] (a14) -> (a7); \draw[red] (a7) -> (a15);
\draw[red] (a16) -> (a8); \draw[red] (a8) -> (a17);
\draw[red] (a18) -> (a9); \draw[red] (a9) -> (a19);
\draw[red] (a20) -> (a10); \draw[red] (a10) -> (a11);

\draw[blue,dotted] (a28) -> (a23); \draw[blue,dotted] (a23) -> (a29);
\draw[blue,dotted] (a29) -> (a24); \draw[blue,dotted] (a24) -> (a30);
\draw[blue,dotted] (a30) -> (a25); \draw[blue,dotted] (a25) -> (a26);
\draw[blue,dotted] (a26) -> (a21); \draw[blue,dotted] (a21) -> (a27);
\draw[blue,dotted] (a27) -> (a22); \draw[blue,dotted] (a22) -> (a28);

\draw[blue] (a11) -> (a27); \draw[blue] (a27) -> (a12);
\draw[blue] (a13) -> (a28); \draw[blue] (a28) -> (a14);
\draw[blue] (a15) -> (a29); \draw[blue] (a29) -> (a16);
\draw[blue] (a17) -> (a30); \draw[blue] (a30) -> (a18);
\draw[blue] (a19) -> (a26); \draw[blue] (a26) -> (a20);

\draw[blue] (a21) to [out=180,in=-90,looseness=1.5] (a25);
\draw[blue] (a25) to [out=100,in=180,looseness=1.5] (a24);
\draw[blue] (a24) to [out=30,in=120,looseness=1.5] (a23);
\draw[blue] (a23) to [out=-45,in=45,looseness=1.5] (a22);
\draw[blue] (a22) to [out=-120,in=-30,looseness=1.5] (a21);

\end{tikzpicture}
  \caption{Vertex-transitive, strongly connected, odd, non-Hamiltonian 2-dd}
  \label{fig-v30}
\end{figure}
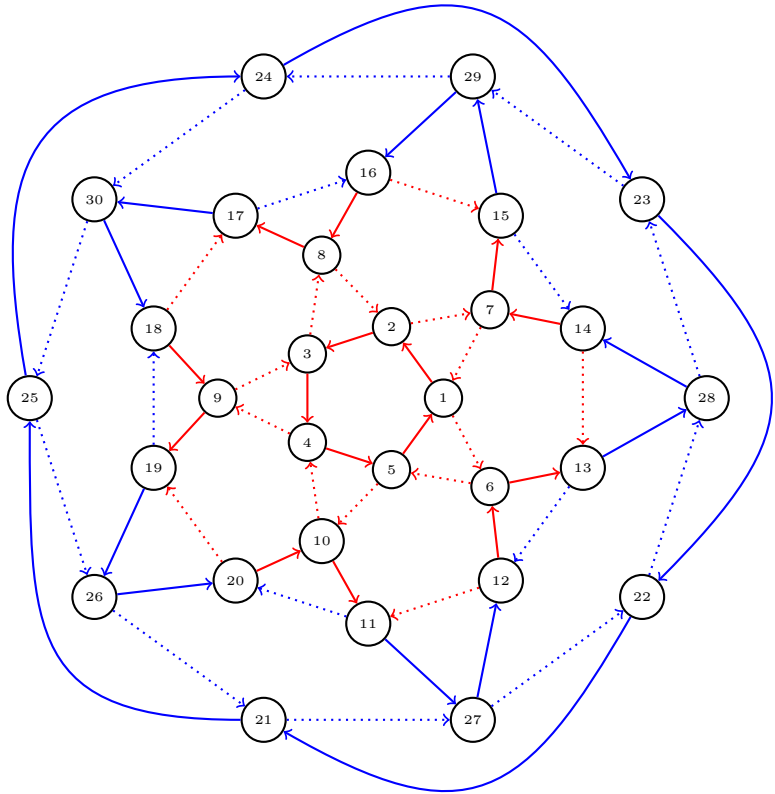

We conclude with a couple of techniques for creating strongly connected non-Hamiltonian
elements of $\fullG$. The first is to take any connected, closed element of $\partialG$
and splice all of its unsaturated vertices with an arbitrary connected element of
$\partialG$ with the same unsaturated vertex count; by Theorem \ref{thm-non-ham}, all
such 2-dds must be non-Hamiltonian. The second is to start with a connected 2-graph
$G_1 \in \partialG$ with a unique open route $r$ and a connected non-Hamiltonian 2-dd
$G_2 \in \fullG$. Split $|r|$ arbitrary vertices of $G_2$ and splice them with the
unsaturated vertices of $G_1$ such that whenever an entry vertex $u$ of $G_1$ is spliced
with $v^{in}$, $r(u)$ is spliced with $v^{out}$. By Theorem
\ref{thm-reduction}\ref{thm-reduction-b}, all such 2-dds must be non-Hamiltonian.

\section{References}
\begin{enumerate}
\item
  {\sl Even, S.} \label{ref-even}
  {\bf Graph Algorithms}, Computer Science Press, Rockville, MD (1979) p 64.
\item
  {\sl Ramanath, M. V. S.} \label{ref-ram}
  {\bf Factors in a class of Regular Digraphs}, J. Graph Theory, vol. 9 (1985)
  p. 161-175.
\item
  {\sl Ramanath, M. V. S.} and {\sl Walsh, T. R.}, \label{ref-ramwalsh}
  {\bf Enumeration and Generation of a Class of Regular Digraphs},
  J. Graph Theory, vol. 11 (1987), p. 471-479.
\item
  {\sl Ramanath, M. V. S.} and {\sl Walsh, T. R.}, \label{ref-ramwalsh-tech}
  {\bf Enumeration and Generation of a Class of Regular Digraphs}, Technical Report
  No. 147 (1986), University of Western Ontario, Canada, p. 28.
\item
  {\sl Tutte, W. T.},  \label{ref-tutte},
  {\bf On Hamiltonian circuits}, Journal of the London Mathematical Society,
  21 (2): 98–101, 1946
\item
  {\sl Meredith, G. H. J.}, \label{ref-meredith},
  {\bf Regular n-valent n-connected nonHamiltonian non-n-edge-colorable graphs},
  Journal of Combinatorial Theory, Series B. 14: 55–60, 1973
\item
  {\sl Trotter, W. T.} and {\sl Erd\"os, P.}, \label{ref-trotter},
  {\bf When the Cartesion product of directed cycles is Hamiltonian},
  J. Graph Theory vol. 2 (1978) 137-142.
\item
  {\sl Charles Delorme}, {\bf Cayley digraphs and graphs}, \label{ref-delorme},
  European Journal of Combinatorics, vol. 34, issue 8 (2013), 1307-1315
\end{enumerate}

\clearpage
\end{flushleft}
\end{document}